\documentclass[11pt,hidelinks]{article}

\usepackage{MyMacros} 
\usepackage{authblk}
\usepackage{mathtools}

\usepackage{calrsfs}
\DeclareMathAlphabet{\pazocal}{OMS}{zplm}{m}{n}
\renewcommand{\cV}{\pazocal{V}}

\newcommand{\Mod}[1]{\ \mathrm{mod}\ #1}

\newcommand{\dif}{\mathop{}\!\mathrm{d}}

\usepackage{comment}
\newcommand{\bdelta}{\bm{\delta}}

\newcommand\nnfootnote[1]{%
  \begin{NoHyper}
  \renewcommand\thefootnote{}\footnote{#1}%
  \addtocounter{footnote}{-1}%
  \end{NoHyper}
}

\title{A new family of smooth copulas with arbitrarily irregular densities}

\author{Micha\"{e}l Lalancette$^*$}
\author{Robert Zimmerman$^*$}
\affil{Department of Statistical Sciences \\
University of Toronto}

\date{\today}

\begin{document}

\maketitle
\nnfootnote{$^*$Equal contribution}

\begin{abstract}
    Copulas are known to satisfy a number of regularity properties, and one might therefore believe that their densities, when they exist, admit a certain degree of regularity themselves. We show that this is not true in general by constructing a broad family of copulas which admit densities that can hardly be considered regular. The copula densities are constructed from arbitrary univariate densities supported on the unit interval, and we show by example that the copula densities can inherit pathological behaviour from the underlying univariate densities. In particular, we construct a nontrivial univariate density which is unbounded in every open set of the unit interval, and show that it induces a copula density which is finite everywhere but unbounded in every neighborhood of the unit hypercube. Nevertheless, all of our copulas are shown to enjoy attractive smoothness properties.
\end{abstract}

\section{Introduction}

Let $d \geq 2$. A $d$-dimensional \emph{copula} is a function $C:[0,1]^d \to [0,1]$ satisfying the following three conditions:
\begin{align}
    C(1, \dots, 1, u, 1, \dots, 1) &= u, \quad 0 \leq u \leq 1 \label{eq:marginal} \\
    C(u_1, \ldots, u_{i-1}, 0, u_{i+1}, \ldots, u_d) &= 0, \quad 0 \leq u_j \leq 1 \label{eq:grounded} \\
    \sum_{i_1 = 1}^2 \cdots \sum_{i_d = 1}^2 (-1)^{i_1 + \cdots i_d} C\left( u_1^{(i_1)}, \dots, u_d^{(i_d)} \right) &\geq 0, \quad 0 \leq u_j^{(1)} \leq u_j^{(2)} \leq 1. \label{eq:dincreasing}
\end{align}
According to a famous result in probability theory known as \emph{Sklar's theorem} \citep{sklar1959fonctions}, for any random vector $\bX \in \R^d$ with continuous marginal distribution functions $F_1, \dots, F_d$, there is a unique copula $C$ that satisfies
\[
    C(\bu) := \Prob(F_1(X_1) \leq u_1, \dots, F_d(X_d) \leq u_d), \quad \bu := (u_1, \cdots, u_d).
\]
The function $C$ is said to be the \emph{copula of} $\bX$; it is the joint distribution function of the random vector $(F_1(X_1), \dots, F_d(X_d))$. Conversely, every distribution function satisfying \cref{eq:marginal} is a copula. Informally, the copula fully characterizes the dependence between the random variables $X_1, \dots, X_d$ in a way that does not depend on their distributions $F_1, \dots, F_d$. An extensive literature \citep{schweizer1991thirty,joe1997multivariate,nelsen2007introduction,durante2016principles} reveals a vast array of dependence structures that can be explicitly described using copulas.

If the distribution of $\bX$ is absolutely continuous with respect to Lebesgue measure, then $C$ admits a \emph{copula density} $c$ given by
\[
    c(\bu) := \frac{\partial^d}{\partial u_1 \dots \partial u_d} C(\bu).
\]
The fact that \cref{eq:marginal,eq:grounded,eq:dincreasing} are stronger than the standard properties of multivariate distribution functions imposes a number of regularity properties. For example, $C$ satisfies sharp lower and upper bounds known as the \emph{Fr\'echet--Hoeffding bounds}: 
\[
    \max\left\{\sum_{h=1}^d u_h - (d-1), 0 \right\} \leq C(\bu) \leq \min\{u_1,\ldots,u_d\}.
\]
It is also not hard to see that $C$ is 1-Lipschitz with respect to the supremum norm. One might therefore suspect that when $C$ is absolutely continuous, its copula density $c$ cannot be too poorly-behaved. In practice, however, one often encounters poor behavior on the boundary of the support of $c$ or on sets of measure zero. For example, the Clayton copula density \citep{hofert2012likelihood} is unbounded along the curve $\sum_{i=1}^d u_i^{-\theta} = d-1$ when $\theta < -1/d$, while the Gaussian copula density is unbounded at the corners $\mathbf{0} := (0, \dots, 0)$ and $\mathbf{1} := (1, \dots, 1)$; other examples include the Gumbel and $t$ families \citep{bouezmarni2013bernstein}. Such behaviour is inconvenient but usually not overly restrictive. However, as we shall demonstrate, there exist highly regular copulas with densities that are poorly behaved on their entire domain.

This paper offers two main contributions. First, in \cref{sub:construction}, we construct a class $\cC$ of absolutely continuous copulas on $[0,1]^d$, for which the corresponding class of copula densities includes highly pathological functions. Each copula $C_f \in \cC$ is ``generated'' by an underlying univariate density $f$ on $[0,1]$ whose smoothness behaviour determines the smoothness of the corresponding copula density $c_f$. In particular, the set of discontinuity points of $f$ determines that of $c_f$. This, among other properties of the family $\cC$, is shown in \cref{sub:properties}. In \cref{sub:simpleexamples}, we exhibit several examples of copulas in $\cC$. As our second main contribution, in \cref{sub:univariate}, we identify a nontrivial copula $C_{f^*} \in \cC$ whose copula density $c_{f^*}$ is unbounded in every neighborhood of the unit hypercube --- that is, $\sup_{\bu \in U} c_{f^*}(\bu) = \infty$ for every nonempty open set $U \subset [0, 1]^d$. Nevertheless, we show that any $C_f \in \cC$, including $C_{f^*}$, must have partial derivatives of higher order that are continuous everywhere as well as bounded, which opens up interesting connections to statistical inference via empirical copulas. We conclude with a brief discussion and suggestions for future work in \cref{sec:discussion}.

Our results, particularly our construction of the aforementioned copula $C_{f^*}$, showcase a fundamental decoupling between copulas and their densities. A certain degree of smoothness is required of a copula $C$ for elegant probabilistic properties to hold, especially with regards to large sample statistical theory, when observations are drawn from a distribution with copula $C$. In contrast, our results show that considering only copulas with ``traditional" densities (which are, for example, bounded and continuous almost everywhere) is indeed restrictive.

\subsection{Notation and conventions}

We write $\bu$ for a vector in $\R^d$ and $u_1, \dots, u_d$ for its elements, $\bu_{-j} = (u_1,\ldots,u_{j-1},u_{j+1},\ldots,u_d) \in \R^{d-1}$ for $\bu$ with its $j$th element removed, and $\bu_J$ for the subvector indexed by the elements of an arbitrary set $J \subset \{1, \dots, d\}$. We use $\be_i$ for the $i$th canonical basis vector. $[\bzero, \bu]$ is to be understood as the rectangle $\bigtimes_{h=1}^d [0,u_h]$. By $x \Mod{1}$, we mean the fractional part of $x$ given by $x - \lfloor x \rfloor$, where $\lfloor x \rfloor$ denotes the largest integer upper bounded by $x$. All densities are taken with respect to Lebesgue measure on $\R^p$, where the dimension $p$ should be clear from context. We define $\cF_{[0,1]}$ as the set of densities on $[0,1]$, i.e., any $f \in \cF_{[0,1]}$ is measurable, non-negative and integrates to 1 on $[0,1]$. To avoid ambiguity, we continue to distinguish ``copula densities'' from ``univariate densities'' throughout. All our results hold for a fixed, but arbitrary, dimension $d \geq 2$, unless specified otherwise.

\section{Construction of \texorpdfstring{$\cC$}{Cf} and basic properties}\label{sec:construction}

\subsection{The construction}\label{sub:construction}

Our construction is very simple and works in any dimension. Rather than explicitly defining a copula $C_f \in \cC$, we begin with an appropriate univariate density and use it to specify the copula density $c_f$; this in turn determines $C_f$. Specifically, given any $f \in \cF_{[0,1]}$, we define
\begin{equation}
    c_f(\bu) := f\left( \sum_{j=1}^d u_j \Mod{1} \right), \quad \bu \in [0, 1]^d. \label{eq:ourcopula}
\end{equation}

Our first result shows that $c_f$ is a true copula density. We thus define our class $\cC$ as the set of $d$-dimensional copulas which have a density in $\{c_f: f \in \cF_{[0, 1]}\}$.

\begin{prop} \label{prop:copula}
    
    For every univariate density $f \in \cF_{[0,1]}$, there exists a copula $C_f$ whose copula density is given by $c_f$.
    
\end{prop}

\begin{proof}
    
    The function $c_f$ is clearly non-negative and measurable, as it is a composition of finitely many measurable functions. In view of \cref{eq:marginal}, it is therefore a valid copula density if for each $j \in \{1, \dots, d\}$, it satisfies the marginal constraint
    \begin{equation} \label{eq:marginaldensity}
        \int_{[0, 1]^{d-1}} c_f(\bu) \dif\bu_{-j} = 1
    \end{equation}
    regardless of the value of $u_j$. In fact, the function $c_f$ satisfies a stronger property: let $i \in \{1, \dots, d\}$ be an arbitrary index and fix all the coordinates of $\bu$ but the $i$th one. Define
    \[
        v := \sum_{j \neq i} u_j \Mod{1}.
    \]
    Then
    \begin{equation} \label{eq:integrateto1}
        \int_0^1 c_f(\bu) \dif u_i = \int_0^{1 - v} f(u_i + v) \dif u_i + \int_{1-v}^1 f(u_i + v - 1) \dif u_i = \int_v^1 f(u) \dif u + \int_0^v f(u) \dif u
        = 1,
    \end{equation}
    where the second equality follows by an affine change of variable. It follows that \cref{eq:marginaldensity} is satisfied; hence $c_f$ is the density of a copula $C_f$.
\end{proof}

\subsection{Some properties of the copulas \texorpdfstring{$C_f$}{Cf}}\label{sub:properties}

It is interesting to note that while the copulas $C_f$ can have extremely irregular densities (see \cref{sub:univariate}), the copulas themselves enjoy several attractive properties which we list here.

\begin{prop}\label{prop:basicproperties}
    
    If $\bU \sim C_f$ for some $f \in \cF_{[0, 1]}$, the following hold.
    \begin{enumerate}
        \item $\bU$ is exchangeable.
        \item For any $i \in \{1, \dots, d\}$, $\bU$ has the following stochastic representation, where $\Pi$ denotes the independence copula:
        \[
            \bU_{-i} \sim \Pi, \quad \bigg( U_i + \sum_{j \neq i} U_j \bigg) \Mod{1}\ \bigg|\ \bU_{-i} \sim f.
        \]
    \end{enumerate}
    
\end{prop}

\begin{proof}
    
    Exchangeability follows immediately from \cref{eq:ourcopula}; the density $c_f(\bu)$ depends only on $\bu$ through the sum of its arguments, which is a symmetric function.
    
    For the stochastic representation, the independence of the random vector $\bU_{-i}$ follows from the fact that its density is given by
    \[
        \int_0^1 c_f(\bu) \dif u_i = 1, \quad \bu_{-i} \in [0, 1]^{d-1}
    \]
    by \cref{eq:integrateto1}. It then follows that given $\sum_{j \neq i} U_j$, the conditional density of
    \[
        \bigg( U_i + \sum_{j \neq i} U_j \bigg) \Mod{1}
    \]
    is given by $f/1 = f$.
\end{proof}

An immediate consequence of the above representation is that if a random vector $\bX$ has copula $C_f$, then the marginal vector $\bX_{-i}$ has the independence copula. That is, any subset of at most $d-1$ of the variables $X_1, \dots, X_d$ are mutually independent. As a corollary, we find that the intersection of $\cC$ and the class of Archimedean copulas is trivial, so no confusion arises from our use of the word ``generator'' and its use for Archimedean generators. 

\begin{coro}

    The only Archimedean copula in $\cC$ is $\Pi$. 

\end{coro}

\begin{proof}
Suppose $\bU$ has an Archimedean copula, so that $C_f(\bu) = \psi\left( \sum_{i=1}^d \psi^{-1}(u_i)\right)$ for some Archimedean generator $\psi$. Such a function has a continuous, decreasing inverse on $[0,1]$ which satisfies $\psi^{-1}(1) =  0$. For any index $i \in \{1,\ldots,d\}$, the distribution of the marginal vector $\bU_{-i}$ is given by \[C_f(u_1,\ldots,u_{i-1},1,u_{i+1},\ldots, u_d) = \psi\left( \sum_{j \neq i} \psi^{-1}(u_j)\right) = \prod_{j \neq i} u_j,\]where the second equality follows from \cref{prop:basicproperties}. Thus $\sum_{j \neq i} \psi^{-1}(u_j) = \psi^{-1}\left( \prod_{j \neq i} u_j\right)$, so that $\psi^{-1}$ satisfies Cauchy's logarithmic functional equation, the continuous and decreasing solutions of which are of the form $\psi^{-1}(t) = -c \log(t)$ for $c > 0$ \citep{aczel1966lectures}. Thus $\psi(t) = \exp(-t/c)$, and it follows that $C_f(\bu) = \prod_{i=1}^d u_i = \Pi(\bu)$.
\end{proof}

Our next result concerns the smoothness of the partial derivatives of $C_f$.

\begin{prop} \label{prop:pd}
    
    For every $f \in \cF_{[0, 1]}$ and each $j \in \{1, \dots, d\}$, the $j$th partial derivative $\dot C_{f, j}$ of $C_f$ exists and is continuous on the set $\cV_j := \{\bu \in [0, 1]^d: u_j \in (0, 1)\}$. If $d \geq 3$, then for each pair $(i, j) \in \{1, \dots, d\}^2$, the second order partial derivative $\Ddot C_{f, ij}$ exists and is continuous on $\cV_i \cap \cV_j$, and $|\Ddot C_{f, ij}(\bu)| \leq 1$.
    
\end{prop}

\begin{rem}
It is a standard fact that wherever they exist, the first order partial derivatives of any copula are bounded by 1. That is, for any copula $C$, we have $0 \leq \dot C_{j}(\bu) \leq 1$ for all $j \in \{1,\ldots,d\}$, which follows directly from the monotonicity and 1-Lipschitz properties. However, the boundedness of \emph{second order} partial derivatives, even assuming their existence, does not hold in general. Consider, for example, a trivariate extension of entry 4.2.9 in the well-known list of one-parameter Archimedean copulas found in Table 4.1 of \cite{nelsen2007introduction}, given by $C_{\theta}(\bu) = u_1 u_2 u_3 \exp\left(-\theta \log(u_1)\log(u_2)\log(u_3)\right)$ where $\theta \in (0,1]$. This is a valid copula when $\theta \leq (3-\sqrt{5})/2$, for then its Archimedean generator satisfies the $3$-monotonicity property \citep[][Theorem 2.2]{mcneil2009multivariate}. An easy calculation shows that for fixed, positive $u_2$ and $u_3$,
\[
    \Ddot C_{\theta, 11}(\bu) \propto u_1^{-1 - \theta\log(u_2)\log(u_3)}
\]
which is unbounded around $u_1=0$.
\end{rem}

Before proving \cref{prop:pd}, we briefly comment on the applicability of the result to statistical inference. Given independent and identically distributed observations with continuous marginals and copula $C$, the \emph{empirical copula} is the canonical nonparametric estimator of $C$. It is given by
\[
    \hat C_n(\bu) := \frac{1}{n} \sum_{k=1}^n \Ind{R_k^1 \leq nu_1, \dots, R_k^d \leq nu_d},
\]
where $R_k^j$ denotes the rank of the $k$th observation of the $j$th variable among all observations of that variable. It is well known that at a point $\bu \in (0, 1)^d$, $\sqrt{n} (\hat C_n(\bu) - C(\bu))$ is asymptotically normal if and only if the partial derivatives of $C$ exist and are continuous at $\bu$. In fact, \cite{segers2012asymptotics} proved that the \emph{empirical copula process} $\sqrt{n}(\hat C_n - C)$ converges in distribution to a Gaussian process in $\ell^\infty([0, 1]^d)$ if and only if each partial derivative $\dot C_j$ exists and is continuous on $\cV_j$ (see Proposition 3.1 therein). \cite{RWZ17} consider the empirical copula process indexed by a class of right-continuous functions of uniformly bounded variation. There again, existence and continuity of the partial derivatives on the sets $\cV_j$ ensures uniform weak convergence of the empirical copula process to a tight Gaussian process (see their Theorem 5). In \cite{berghaus2017weak}, in addition to the existence and continuity of the first order partial derivatives, the second order partial derivatives are assumed to exist, to be continuous and to satisfy a certain bound. The empirical copula process is then shown to converge in distribution to a tight Gaussian process in weighted metrics, strengthening the result of \cite{segers2012asymptotics} (see their Theorem 2.2).

\cite{segers2017empirical} introduce smoothed versions of the empirical copula. Asymptotic normality of the so-called \emph{empirical beta copula} and \emph{empirical Bernstein copula} processes is shown to hold if the first order partial derivatives exist, are continuous everywhere, and satisfy a local Lipschitz property (see Theorem 3.6 therein). The latter holds, for instance, if the second order partial derivatives exist and are bounded everywhere.

\Cref{prop:pd} guarantees that all of the aforementioned results hold when the data is distributed according to any copula in $\cC$, regardless of the generator --- in particular, they hold for the copula $C_{f^*}$ with nowhere continuous and nowhere bounded density, as constructed in \cref{sub:univariate}. Interestingly, multiple copulas appear smoother than $C_{f^*}$ at first sight, but do not have continuous partial derivatives. For example, the so-called \emph{checkerboard copula} on $[0, 1]^2$ has density $c := 2\One_{[0, 1/2]^2 \cup [1/2, 1]^2}$, which is bounded and almost everywhere continuous. Yet its partial derivatives are not continuous everywhere, so asymptotic normality of the associated empirical copula process fails.

The key to the proof of \cref{prop:pd} is the following.

\begin{lemm} \label{lemm:continuity}
    
    Let $f \in \cF_{[0, 1]}$ and $J \subseteq \{1, \dots, d\}$ be nonempty. Then for any values of $\bu_J \in \R^{|J|}$ and $\bv_{-J} \in \R^{d - |J|}$,
    \[
        \int_{[0, \bu_J]} f\left( \sum_{i=1}^d v_i \Mod{1} \right) \dif \bv_J = \int_{[0, \bu_J]} f\left( \sum_{i \in J} v_i + \sum_{i \notin J} v_i \Mod{1} \right) \dif \bv_J
    \]
    is continuous, both in $\bu_J$ and in $\bv_{-J}$.
    
\end{lemm}

\begin{proof}
    
    First note that in the proof of \cref{prop:copula}, we show that
    \[
        f\left( \sum_{i=1}^d v_i \Mod{1} \right)
    \]
    is integrable on $[0, 1]^d$. By using the same strategy, one can show that it is also integrable on arbitrary compact subsets of $\R^d$. The same can be said if we are integrating only with respect to the subset of variables $\bv_J$. Continuity in $\bu_J$ thus follows from continuity from above of Lebesgue measure.
    
    To establish continuity in the direction of $\bv_{-J}$, we show that an infinitesimal change in $\bv_{-J}$ is equivalent to an infinitesimal change in some component of $\bu_J$; we then simply apply the first part of the present result. Indeed, for a small vector $\bdelta \in \R^d$, adding $\bdelta_{-J}$ to $\bv_{-J}$ shifts $\sum_{i=1}^d v_i$ to
    \[
        \sum_{i \notin J} (v_i + \delta_i) + \sum_{i \in J} v_i = \sum_{i \neq i^*} v_i + (v_{i^*} + \delta),
    \]
    for some arbitrary $i^* \in J$ and for $\delta := \sum_{i \notin J} \delta_i$. Then
    \begin{align*}
        \int_{[0, \bu_J]} f\left( \sum_{i \notin J} (v_i + \delta_i) + \sum_{i \in J} v_i \Mod{1} \right) \dif \bv_J &= \int_{[0, \bu_J]} f\left( \sum_{i \neq i^*} v_i + (v_{i^*} + \delta) \Mod{1} \right) \dif \bv_J
        \\
        &= \int_{[\delta \be_{i^*}, \bu_J + \delta \be_{i^*}]} f\left( \sum_{i=1}^d v_i \Mod{1} \right) \dif \bv_J
        \\
        &\too \int_{[0, \bu_J]} f\left( \sum_{i=1}^d v_i \Mod{1} \right) \dif \bv_J
    \end{align*}
    as $\delta \to 0$, by the first part of the result.
\end{proof}

\begin{proof}[Proof of \cref{prop:pd}]We consider the first and second order partial derivatives separately. 
    
    \textbf{First order partial derivatives.} By exchangeability, we may without loss of generality take $j=1$. Let $\bu \in \cV_1$. By Tonelli's theorem and the fundamental theorem of calculus,
    \[
        \dot C_{f, 1}(\bu) := \frac{\partial}{\partial u_1} \int_0^{u_1} \int_{[0, \bu_{-1}]} c_f(v_1, \bv_{-1}) \dif \bv_{-1} \dif v_1 = \int_{[0, \bu_{-1}]} c_f(u_1, \bv_{-1}) \dif \bv_{-1},
    \]
    where we write $c_f(u_1, \bv_{-1})$ for $c_f$ evaluated at the vector $(u_1, \bv_{-1})$. Note that the application of the fundamental theorem of calculus in the second equality requires continuity of the right-hand side in $u_1$. This is a consequence of \cref{lemm:continuity}, which at the same time asserts that $\dot C_{f, 1}$ is continuous at $\bu$.

    \textbf{Second order partial derivatives.} Now suppose $d \geq 3$. Again by exchangeability, it suffices to treat $\Ddot C_{f, 12}$ and $\Ddot C_{f, 11}$. Using the expression for $\dot C_{f, 1}$ above,
    \[
        \Ddot C_{f, 12}(\bu) := \frac{\partial}{\partial u_2} \int_0^{u_2} \int_{\left[ 0, \bu_{-\{1, 2\}} \right]} c_f(u_1, v_2, \bv_{-\{1, 2\}}) \dif \bv_{-\{1, 2\}} \dif v_2 = \int_{\left[ 0, \bu_{-\{1, 2\}} \right]} c_f(u_1, v_2, \bv_{-\{1, 2\}}) \dif \bv_{-\{1, 2\}},
    \]
    which is continuous in $\bu$. As before, we have used Tonelli's theorem, the fundamental theorem of calculus and \cref{lemm:continuity}.
    
    We now evaluate $\Ddot C_{f, 11}$. Let $u \in \cV_1$, $\delta \in \R$. We have
    \begin{align*}
        \dot C_{f, 1}(\bu + \delta \be_1) &= \int_0^{u_2} \int_{\left[ 0, \bu_{-\{1, 2\}} \right]} c_f(u_1 + \delta, v_2, \bv_{-\{1, 2\}}) \dif \bv_{-\{1, 2\}} \dif v_2
        \\
        &= \int_0^{u_2} \int_{\left[ 0, \bu_{-\{1, 2\}} \right]} c_f(u_1, v_2 + \delta, \bv_{-\{1, 2\}}) \dif \bv_{-\{1, 2\}} \dif v_2
        \\
        &= \int_\delta^{u_2 + \delta} \int_{\left[ 0, \bu_{-\{1, 2\}} \right]} c_f(u_1, v_2, \bv_{-\{1, 2\}}) \dif \bv_{-\{1, 2\}} \dif v_2.
    \end{align*}
    In case $\delta < 0$ or $u_2 + \delta > 1$, the definition of $c_f$ in \cref{eq:ourcopula} can be trivially extended outside $[0, 1]^d$. Deduce that $\Ddot C_{f, 11}(\bu)$ exists and is given by
    \[
        \lim_{\delta \to 0} \frac{\dot C_{f, 1}(\bu + \delta \be_1) - \dot C_{f, 1}(\bu)}{\delta} = \int_{\left[ 0, \bu_{-\{1, 2\}} \right]} c_f(u_1, u_2, \bv_{-\{1, 2\}}) \dif \bv_{-\{1, 2\}} - \int_{\left[ 0, \bu_{-\{1, 2\}} \right]} c_f(u_1, 0, \bv_{-\{1, 2\}}) \dif \bv_{-\{1, 2\}},
    \]
    which is continuous in $\bu$ by \cref{lemm:continuity}. As discussed in the proof thereof, $c_f(u_1, u_2, \cdot)$ is a density on $[0, 1]^{d-2}$. It follows that $\Ddot C_{f, 12}(\bu) \in [0, 1]$ and $\Ddot C_{f, 11}(\bu) \in [-1, 1]$, which completes the proof.
\end{proof}

\begin{rem}
In proving \cref{prop:pd}, we find that no matter how ``rough'' $c_f$ is, as long as it is integrated with respect to at least one variable, the resulting antiderivative is continuous. Since differentiating the copula essentially removes one integral, the resulting derivative exists and is continuous as long as at least one integral remains. By iterating the process, one can show that all the partial derivatives of $C_f$ of order up to $d-1$ exist and that they are continuous and bounded on suitable domains.
\end{rem}

We now examine a traditional measure of concordance for bivariate copulas, Spearman's $\rho$. It is certain that no copula in $\cC$ can have Spearman's $\rho$ too close to 1, since no such copula approaches the comonotonicity copula. The countermonotonicity copula, however, can be approached by choosing a generator $f$ whose support approaches $\{0, 1\}$. In fact, the range of values of Spearman's $\rho$ for the copulas in $\cC$ is found to be exactly $(-1, 1/2)$.

\begin{prop}\label{prop:spearman}
    
    If $d=2$, then for every $f \in \cF_{[0, 1]}$, the copula $C_f$ has
    \[
        \rho_{C_f} = 6\EE_{X \sim f}\left[X(1-X)\right] - 1 \in \left(-1, \frac{1}{2}\right),
    \]
    where both limiting values $-1$ and $1/2$ can be approached.
    
\end{prop}

\begin{proof}
    
    According to Theorem 5.1.6 in \cite{nelsen2007introduction},
    \begin{equation} \label{eq:rho}
        \rho_{C_f} = 12 \int_{[0, 1]} u_1 u_2 c_f(\bu) \dif \bu - 3.
    \end{equation}
    The integral can be written as
    \begin{align*}
        &\int_0^1 \int_0^1 u_1 u_2 f(u_1 + u_2 \Mod 1) \dif u_2 \dif u_1
        \\
        &\quad = \int_0^1 u_1 \left\{ \int_0^{1-u_1} u_2 f(u_1 + u_2) \dif u_2 + \int_{1-u_1}^1 u_2 f(u_1 + u_2 - 1) \dif u_2 \right\} \dif u_1
        \\
        &\quad = \int_0^1 u_1 \left\{ \int_{u_1}^1 (u_2 - u_1) f(u_2) \dif u_2 + \int_0^{u_1} (u_2 - u_1 + 1) f(u_2) \dif u_2 \right\} \dif u_1
        \\
        &\quad = \int_0^1 u_1 \dif u_1 \int_0^1 u_2 f(u_2) \dif u_2 - \int_0^1 u_1^2 \dif u_1 \int_0^1 f(u_2) \dif u_2 + \int_0^1 F(u_1) \dif u_1
        \\
        &\quad = \frac{1}{2} \int_0^1 u f(u) \dif u - \frac{1}{2} \int_0^1 u^2 f(u) \dif u + \frac{1}{6},
    \end{align*}
    where we have used the fact that for a positive integer $p$,
    \[
        p \int_0^1 u^{p-1} (1 - F(u)) \dif u = \int_0^1 u^p f(u) \dif u
    \]
    which is the $p$th moment of the distribution $f$. Inserting this into \cref{eq:rho} yields
    \[
        \rho_{C_f} = 6\int_0^1 u(1-u) f(u) \dif u - 1,
    \]
    which is the desired expression.
    
    For the bounds, we note that $0 \leq x(1-x) \leq 1/4$, and that these bounds are attained at $x \in \{0, 1\}$ and at $x=1/2$, respectively. Letting the density $f$ put all of its mass arbitrarily close to $0$ or $1$ will yield $\EE_{X \sim f}\left[X(1-X)\right]$ arbitrarily small, whereas if $f$ has all its mass arbitrarily close to $1/2$, that quantity will approach its maximum of $1/4$.
\end{proof}

For a random vector with copula $C_f$, no subvector has positive tail dependence. There is, however, no limit to the amount of negative tail dependence that it can have; for instance, if $f$ assigns no mass to a neighborhood of $0$, then $C_f$ has perfect negative lower tail dependence. The following result is stated in terms of lower tail dependence; an analogous statement holds for the upper tail.

\begin{prop} \label{prop:tail}
    
    For every $f \in \cF_{[0, 1]}$, we have
    \[
        C_f(t, \dots, t) = o\big( t^{d-1} \big), \quad t \downarrow 0.
    \]
    
\end{prop}

\begin{proof}
    
    For simplicity, assume that $t \leq 1/d$. First, note that the hypercube $[0, t]^d$ is included in the region $[0, td] \cdot \Delta^{d-1} = \{r\Delta^{d-1}: 0 \leq r \leq td\}$, where $\Delta^{d-1} := \{\bu \in [0, 1]^d: \sum_j u_j = 1\}$ is the unit simplex. Thus
    \[
        C_f(t, \dots, t) \leq \int_{[0, td] \cdot \Delta^{d-1}} c_f(\bu) \dif \bu = \int_{[0, td] \cdot \Delta^{d-1}} f\left( \sum_{j=1}^d u_j \right) \dif \bu = \int_0^{td} f(r) \int_{r\Delta^{d-2}} \dif\bv \dif r,
    \]
    where the last equality comes from the change of variable $r = \sum_j u_j$, $\bv = \bu_{-1}$. The inner integral is $r^{d-1}$ times the volume under the $(d-2)$-dimensional simplex in $[0, 1]^{d-1}$, which is equal to $1/(d-1)!$. Recalling that $f$ is a density and hence integrable around $0$, we conclude that
    \[
        C_f(t, \dots, t) \leq \frac{1}{(d-1)!} \int_0^{td} r^{d-1} f(r) \dif r \leq \frac{d^{d-1}}{(d-1)!} t^{d-1} \int_0^{td} f(r) \dif r = o\big( t^{d-1} \big).
    \]
\end{proof}

\section{Examples of generators}\label{sec:examples}

\subsection{Simple examples}\label{sub:simpleexamples}

To begin with, it is obvious from \cref{eq:ourcopula} that by choosing the generator $f_1 := 1$, we obtain the density $c_{f_1} = 1$ of $\Pi$, as shown in the left panel of \cref{fig:cop_pw}. More generally, we can define a sequence of univariate densities $\{f_n\}$ by
\[
    f_n(x) := \sum_{j=1}^{n} \frac{2j-1}{n} \One_{\left[(j-1)/n, j/n \right]}(x).
\]
Each $f_n$ is piecewise constant, placing uniform mass within each interval $[(j-1)/n, j/n]$. The center panel of \cref{fig:cop_pw} shows $c_{f_{10}}$ for the bivariate case. Generally, when $n \geq 2$ the piecewise constancy of $f_n$ manifests itself in $c_{f_n}$ as $n$ pairs of isosceles trapezoids, each perpendicular to the $z$-axis and separated from one another by a distance of $1/n$.

\begin{figure}[ht]
    \centering
    \includegraphics[scale=0.4]{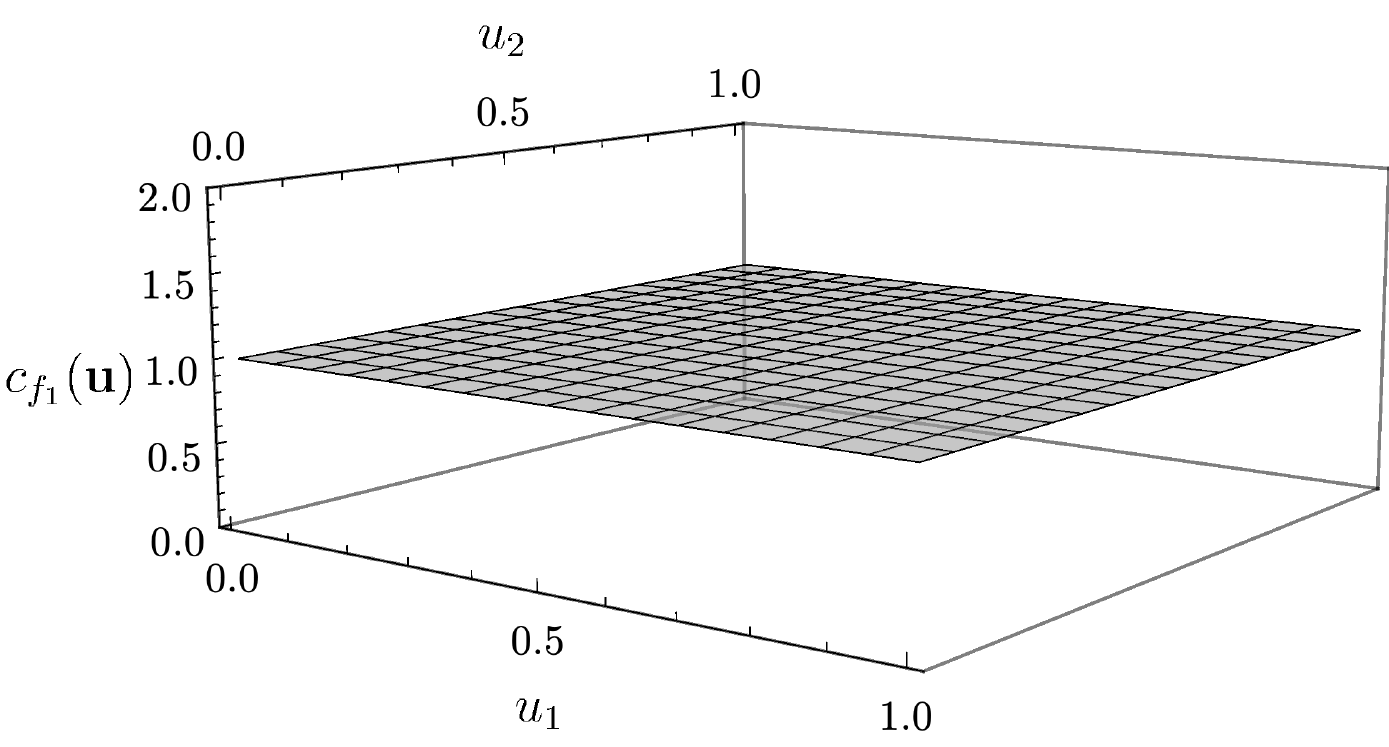}
    \includegraphics[scale=0.4]{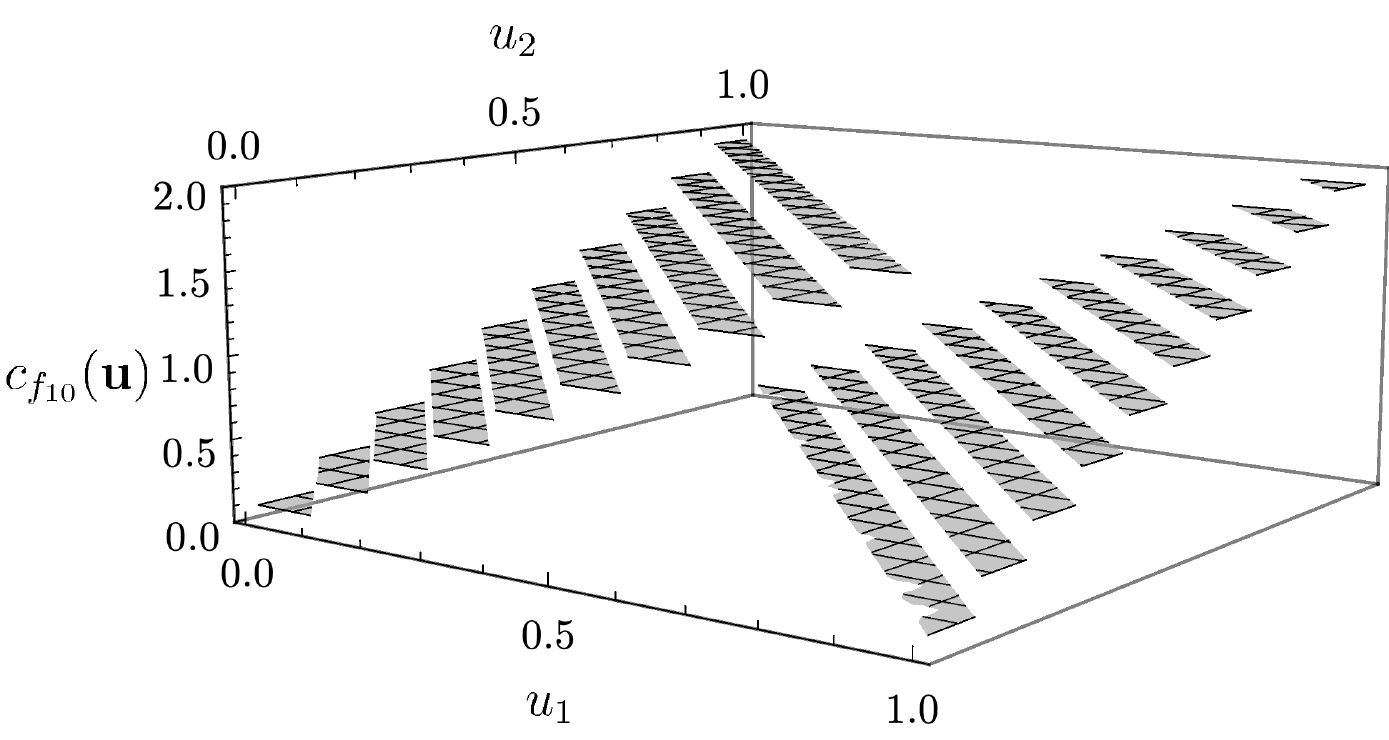}
    \includegraphics[scale=0.4]{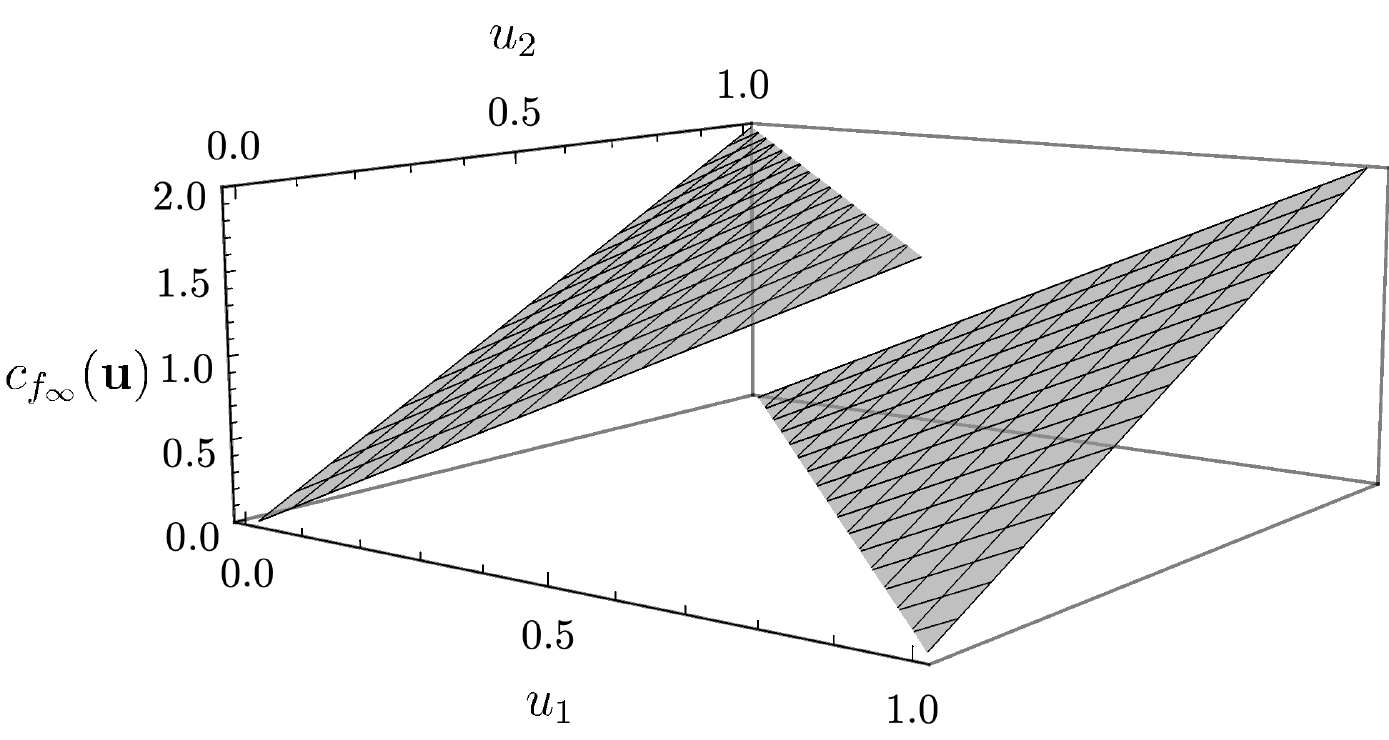}
    \caption{Bivariate copula densities corresponding to the generators $f_1$, $f_{10}$, and $f_\infty.$}
    \label{fig:cop_pw}
\end{figure}

Note that $c_{f_n} \to c_{f_\infty}$ pointwise, where $f_\infty(x) := \lim_{n \to \infty} f_n(x) = 2x$, a special case of the triangular distribution density. By Scheff\'{e}'s lemma, the sequence of copulas $C_{f_n}$ converges weakly to $C_{f_\infty}$. The copula density $c_{f_{\infty}}$ is shown in the right panel of \cref{fig:cop_pw}.

For examples of more ``typical'' generators, we consider the $\text{Beta}(\alpha,\beta)$ distribution for several pairs of parameters. When $\alpha = \beta = 3/2$, the density is
\[
    f_{\frac{3}{2},\frac{3}{2}}(x) = \frac{8\sqrt{x(1-x)}}{\pi}.
\]
This density has zeros at 0 and 1, and so the corresponding copula density is zero on hyperplanes of the form $\{\bu \in [0, 1]^d: \sum_j u_j = n\}$, where $n \in \{0,1,\ldots,d\}$. In the bivariate case, this corresponds to zeros along the line segment $\Delta^1$ as well as at the two isolated points $\bzero$ and $\bone$, as shown in the left panel of \cref{fig:cop_beta}. On the other hand, with $\alpha = \beta = 1/2$, the univariate density is
\[
    f_{\frac{1}{2},\frac{1}{2}}(x) = \frac{1}{\pi \sqrt{x(1-x)}}.
\]
The corresponding copula density is essentially the reciprocal of the previous situation, with singularities along the same line segment and isolated points, as shown in the center panel of \cref{fig:cop_beta}. Finally, we can introduce asymmetry into the copula density by choosing any $\alpha \neq \beta$. For example, $\alpha = 1/2$, $\beta = 3/2$, yields the univariate density
\[
    f_{\frac{1}{2},\frac{3}{2}}(x) = \frac{2}{\pi}\sqrt{\frac{1-x}{x}}.
\]
and the copula density is shown in the right panel of \cref{fig:cop_beta}.

\begin{figure}[ht]
    \centering
    \includegraphics[scale=0.4]{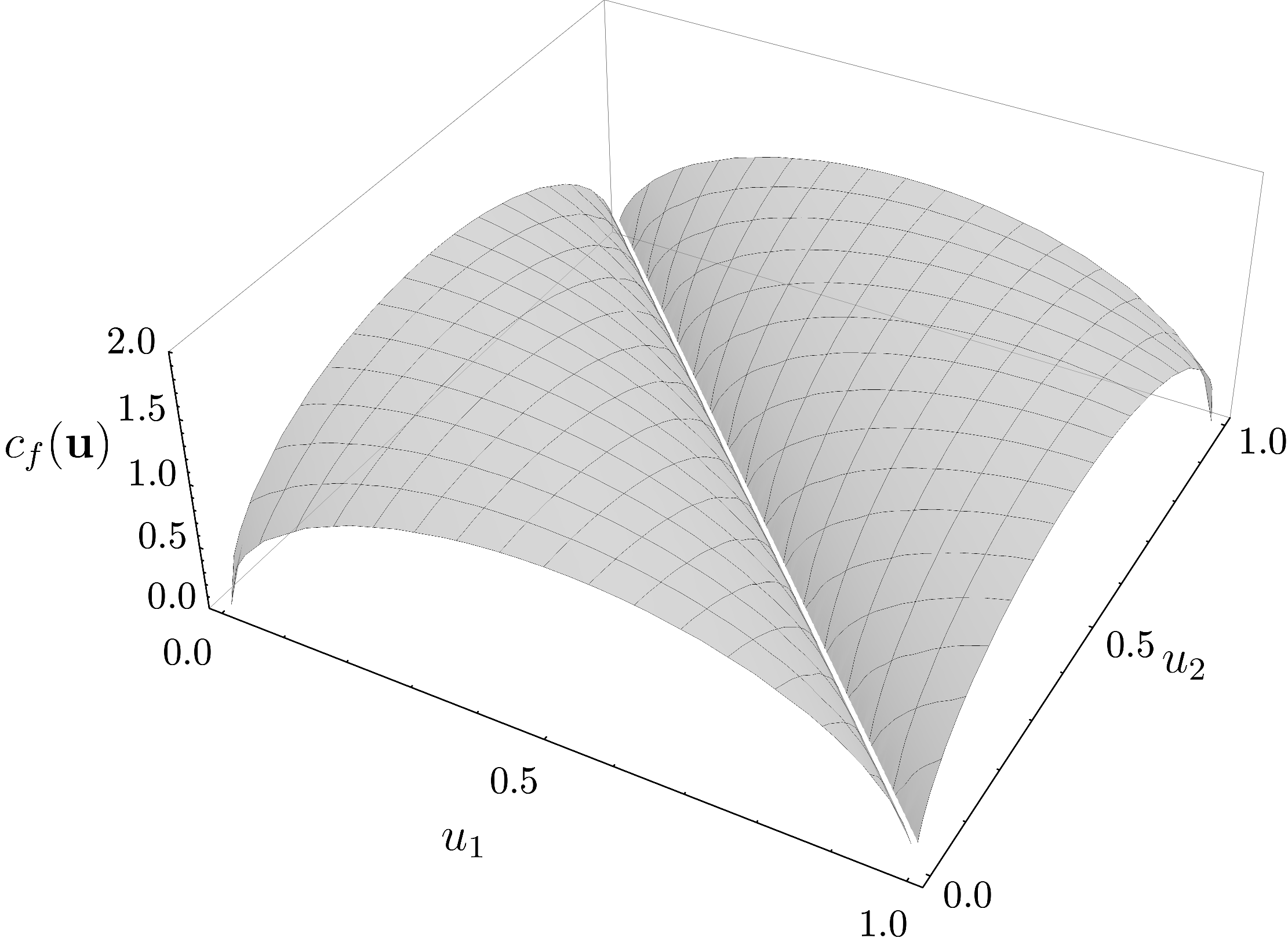}
    \includegraphics[scale=0.4]{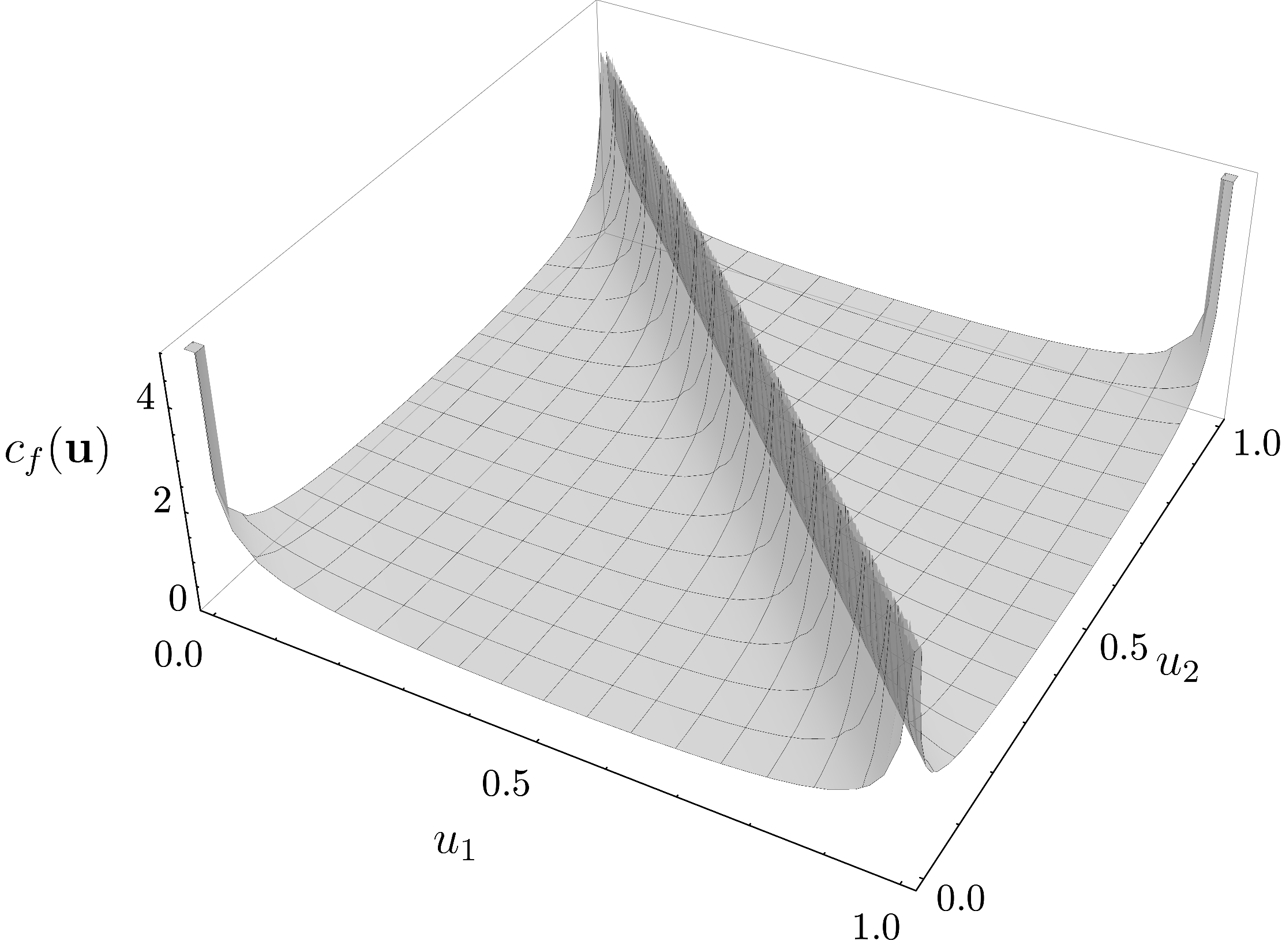}
    \includegraphics[scale=0.4]{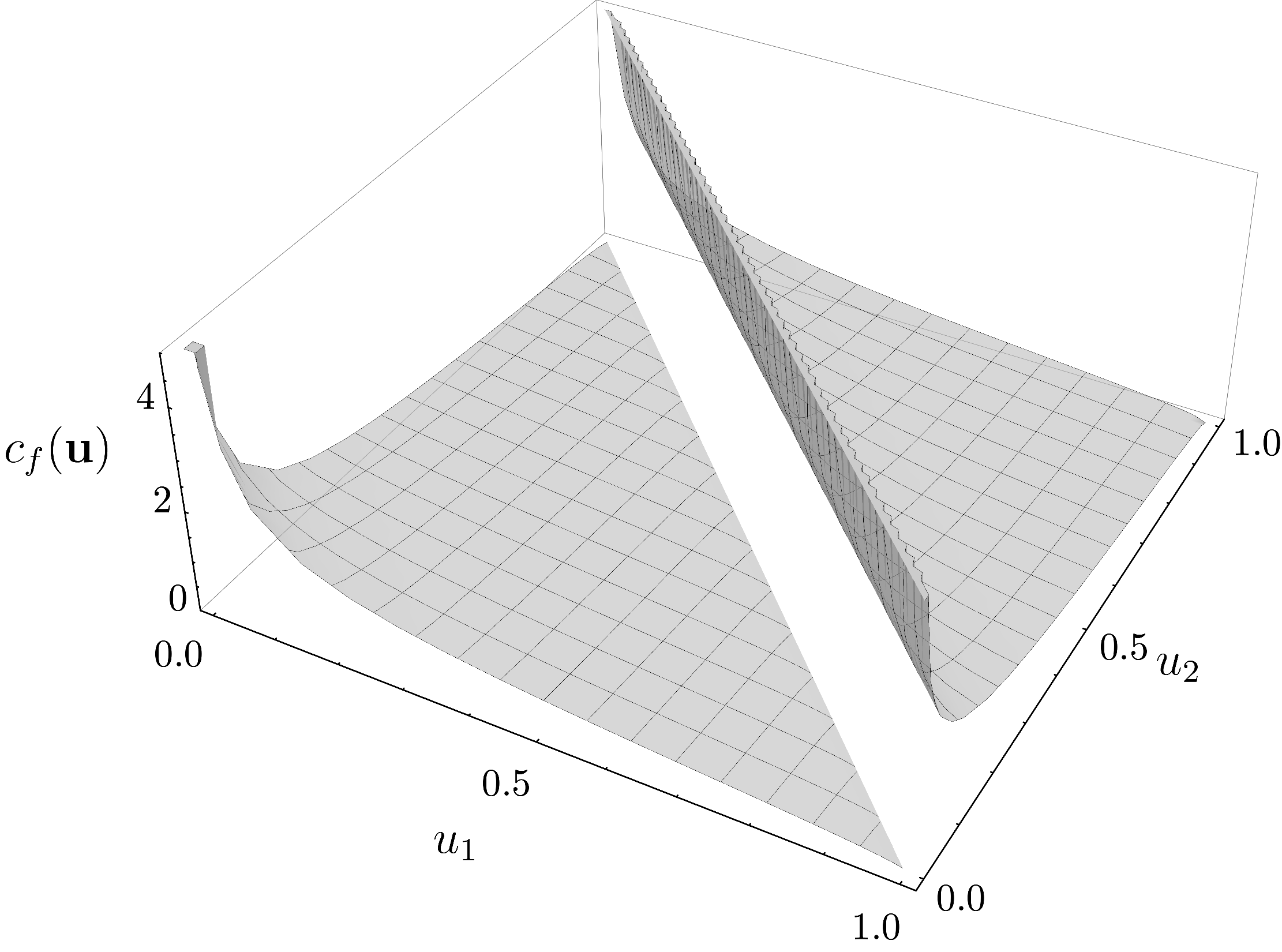}
    \caption{Bivariate copula densities corresponding to the generators $f_{\frac{3}{2},\frac{3}{2}}$, $f_{\frac{1}{2},\frac{1}{2}}$, and $f_{\frac{1}{2},\frac{3}{2}}$. }
    \label{fig:cop_beta}
\end{figure}

Inspection of the plots in \cref{fig:cop_pw,fig:cop_beta} reveals additional ``graphical'' properties of our bivariate copula densities, which readily extend to arbitrary dimensions $d \geq 3$. It is apparent from \cref{eq:ourcopula} that the copula density is constant on any hyperplane of the form $\{\bu \in [0,1]^d: \sum_j u_j = r\}$ where $r \in (0,d)$, on which the copula density is equal to $f(r \Mod{1})$. Moreover, one can show that the projection of the graph $\{\left(\bu, c_f(\bu)\right): \bu \in (0,1)^d\} \subseteq \R^{d+1}$ onto the 2-dimensional subspace $\{\bx \in \R^{d+1}: x_1 = \cdots = x_d\}$ returns $d$ copies of the graph of $f$ arranged side-by-side, each compressed horizontally by a factor of $\sqrt{d}$.

\subsection{A poorly behaved univariate density}\label{sub:univariate}

We now construct a density $f^* \in \cF_{[0,1]}$ which, in contrast with the examples of \cref{sub:simpleexamples}, is quite pathological. To this end, let $\QQ_1 := \QQ \cap (0, 1)$ and $q_1, q_2,\ldots$ be an arbitrary enumeration of $\QQ_1$. For each $q \in \QQ_1$, we define densities $f_q^+$ and $f_q^-$ that diverge at $1-q$ by
\begin{equation}
	f_q^\pm(u) := \begin{cases}\dfrac{1}{2\sqrt{(\pm(u + q))\Mod{1}}}, & u \neq 1-q \\
	1/2, & u = 1-q\\ \end{cases}
\label{eq:f_q}\end{equation}
and $f_q := (f_q^+ + f_q^-)/2$. By a simple change of variable, we see that
\begin{equation} \label{eq:integralfq}
	\int_0^1 f_q^+(u) \dif u = \int_0^{1-q} \frac{1}{2\sqrt{u + q}} \dif u + \int_{1-q}^1 \frac{1}{2\sqrt{u + q - 1}} \dif u = \int_q^1 \frac{1}{2\sqrt{u}} \dif u + \int_0^q \frac{1}{2\sqrt{u}} \dif u = 1,
\end{equation}
and similarly $\int_0^1 f_q^-(u) \dif u = 1$, so $f_q$ is a proper density. Now, let weights $w_q > 0$ be given such that $\sum_{q \in \QQ_1} w_q = 1$. Valid choices include, for example, $w_{q_n}^{(1)} := 6/(\pi^2 n^2)$ and $w_{q_n}^{(2)} := 2^{-n}$. Define
\[
	h^* := \sum_{q \in \QQ_1} w_q f_q.
\]
The function $h^*$ is non-negative and measurable, and it has Lebesgue integral equal to 1 over $[0,1]$ by Tonelli's theorem; it is thus a proper density (although it is not Riemann integrable, as we shall see shortly). To provide intuition, several normalized partial sums of $h^*$ are shown in \cref{fig:fplot}. Each $h_n^*(u)$ is constructed by choosing for $q_1,\ldots,q_n$ the $n$ rationals evenly spaced within the interval $[1/(n+1), n/(n+1)]$, endpoints included. The weights are chosen as $w_{q_j} \propto 1.1^{-\pi(j)}$, $j=1,\ldots,n$ for a random permutation $\pi$ of $(1,\ldots,n)$.

\begin{figure}[ht]
    \centering
    \includegraphics[scale=0.4]{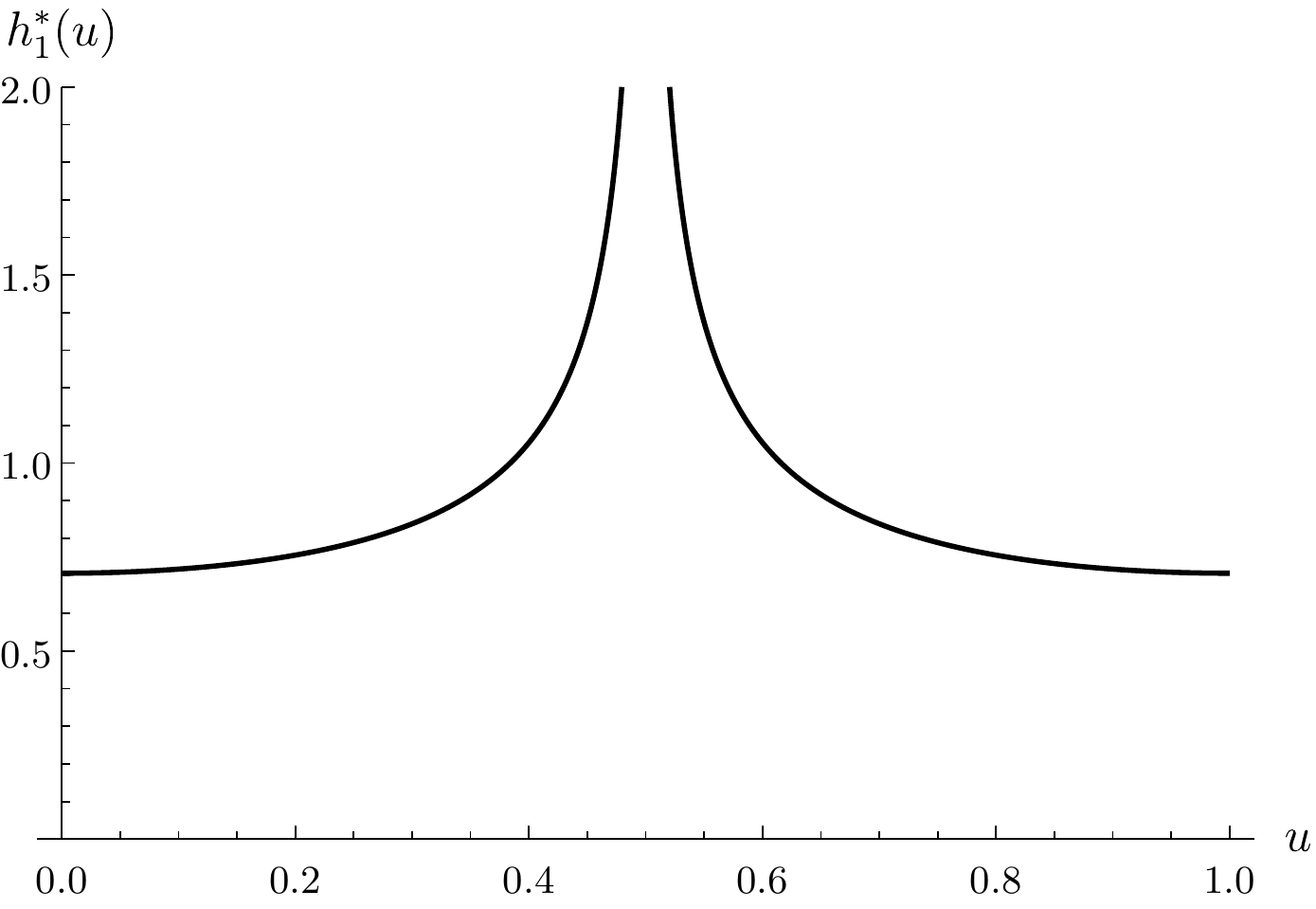}
    \includegraphics[scale=0.4]{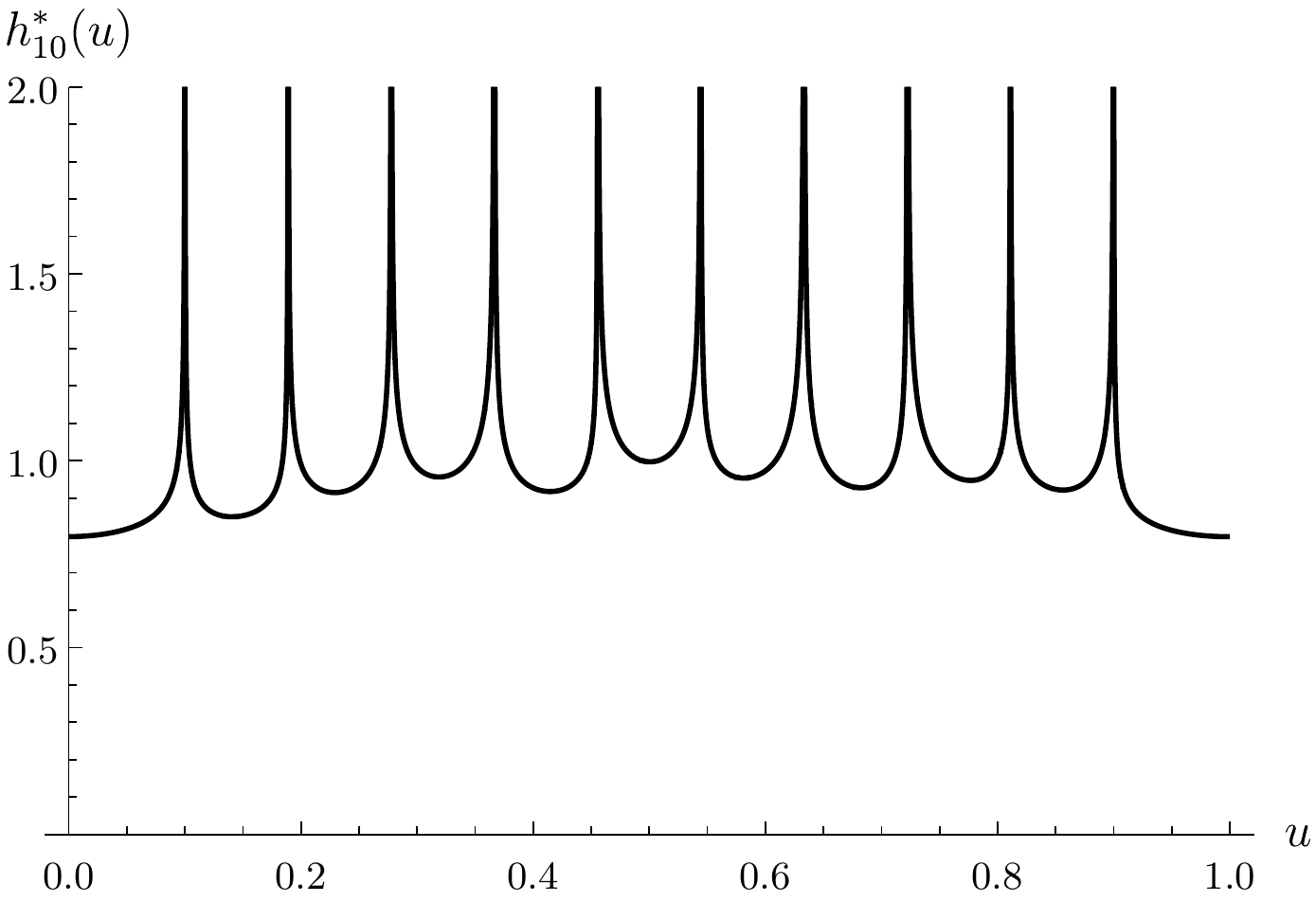}
    \includegraphics[scale=0.4]{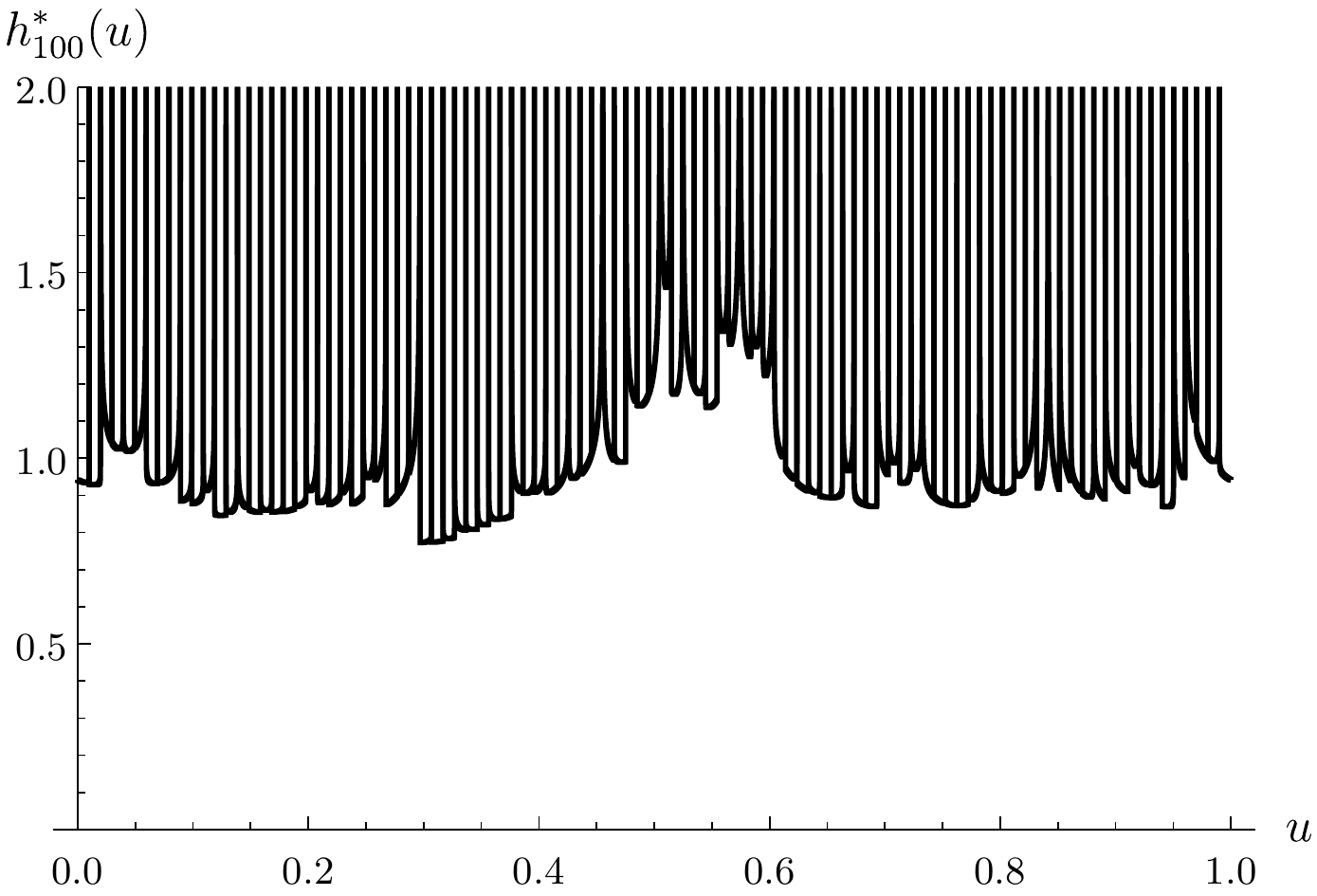}
    \caption{From left to right, the normalized partial sums $h_1^*$, $h_{10}^*$ and $h_{100}^*$.}
    \label{fig:fplot}
\end{figure}

Witness to the irregularity of the associated copula density is the set of singular points of $h^*$, which we investigate here. Consider the set $A = \{x \in [0,1]: h^*(x)= \infty\}$. While $\QQ_1$ is clearly contained in $A$, the latter set is, in fact, much larger.

\begin{prop}
$A$ is uncountable.
\end{prop}
\begin{proof}
Let 
\[
    B_{n,k} := \left(1-q_n, \min\left\{1 - q_n + \frac{w_{q_n}^{2k}}{4}, 1\right\}\right), \quad B := \bigcap_{k \geq 1} \bigcup_{n \geq 1} B_{n,k}.
\]
First, suppose for a contradiction that $B$ is countable. Let $\epsilon > 0$ and fix $y \in (0,1)$. For any $n^* \in \N$ such that $1-q_{n^*} \in (y-\epsilon, y)$, we have $|y - y'| < \epsilon$ for every $y' \in (1-q_{n^*}, \min\{1 - q_{n^*} + \min\{\epsilon, w_{q_{n^*}}^{2k}/4\},1\}) \subseteq B_{{n^*},k}$. It follows that $\cup_{n \geq 1} B_{n,k}$ is an open, dense subset of $[0,1]$. Equivalently, its complement $\cap_{n \geq 1} B_{n,k}^c$ is closed and nowhere dense, and so $B^c = \cup_{k \geq 1} \cap_{n \geq 1} B_{n,k}^c$ is meagre. Now, $B$ is also meagre because it can be written as a countable union of singletons. Therefore the union $B \cup B^c = [0,1]$ is itself meagre, but this contradicts the well known Baire category theorem \citep[see Theorem 3.6 of][for instance]{levy2012basic}. Hence $B$ is uncountable.

We now show that $B \subseteq A$, from which the result will follow. If $x \in B$, then for all $k \geq 1$ there exists some $n(k) \geq 1$ such that $x \in B_{n(k), k}$. That is, the sets $D_k^x := \{n: x \in B_{n, k}\}$, $k \geq 1$, are all nonempty. Then
\[
    2 h^*(x) > \sum_{n \in \bigcup_{k \in \N} D_k^x} w_{q_n} f_{q_n}^+(x) \geq \sum_{n \in \bigcup_{k \in \N} D_k^x} \frac{w_{q_n}}{2\sqrt{(x+q_n) \Mod{1}}} > \sum_{n \in \bigcup_{k \in \N} D_k^x} \frac{w_{q_n}}{2\sqrt{w_{q_n}^2/4}} = \left| \bigcup_{k \in \N} D_k^x \right| = \infty,
\]
and so $x \in A$. To obtain the last equality, suppose for contradiction that the cardinality of the union of the sets $D_k^x$ is finite. Then, since the $D_k^x$ are all nonempty, there must exist an integer $n_0 \in \limsup_{k} D_k^x$. This means that for infinitely many integers $k$,
\[
    1 < x + q_{n_0} < 1 + \frac{w_{q_{n_0}}^{2k}}{4},
\]
a contradiction.
\end{proof}

This is, of course, not the only way to construct a univariate density which is unbounded at uncountably many points. Given any uncountable set $D \subset [0,1]$ of measure zero --- of which the Cantor set is the classic example, although unlike $A$, it is nowhere dense in $(0,1)$ --- one can somewhat artificially define a function on $[0,1]$, such as $h := 1/\One_{D^c}$, which is infinite on $D$ but remains constant at 1 on $D^c$. While any such function is a density in $\cF_{[0,1]}$, it produces only the standard uniform distribution upon integration, and hence its use as a generator simply returns $C_h = \Pi$. More generally, one can produce a density in $\cF_{[0,1]}$ which is well-behaved on $D^c$ by choosing, say, some continuous $f \in \cF_{[0,1]}$ and then defining $h := f/\One_{D^c}$. The copulas generated by such $h$ will not differ from the kinds discussed in \cref{sub:simpleexamples}; their copula densities $c_h$ will not display divergent behavior anywhere. In contrast, the density of $C_{h^*}$ is fundamentally irregular, as it \emph{diverges} in any neighborhood: one can find a sequence converging to any point in $[0, 1]^d$ along which $c_{h^*}$ is finite but $c_{h^*} \to \infty$.

To avoid functions taking on the value $\infty$, we define 
\[
    f^* := h^* \One_{A^c},
\]
which is finite on $[0,1]$ but nevertheless unbounded in every neighborhood. These properties carry on directly to the copula density $c_{f^*}$. Indeed, since $\bu \mapsto \sum_j u_j \Mod 1$ is an open map, for any nonempty open $U \subset [0, 1]^d$, $\sup_{\bu \in U} c_{f^*}(\bu) = \sup_{x \in \tilde U} f^*(x) = \infty$ for some nonempty open $\tilde U \subset [0, 1]$. Still, the finiteness of $f^*$ ensures that $c_{f^*}$ itself is finite at every point.

\section{Discussion}\label{sec:discussion}

Although properties of the family $\cC$ such as the periodicity of the densities along line segments make it an unlikely choice for modelling purposes, the existence of this family should serve to illustrate the flexibility of copulas. It is easy to construct copulas (such as the checkerboard copula) whose densities are continuous almost everywhere but whose partial derivatives fail to exist on certain sets. The copula $C_{f^*}$, introduced in \cref{sub:univariate}, illustrates the opposite situation: the copula is $(d-1)$-times continuously differentiable everywhere, while its density $c_{f^*}$ violates virtually every regularity or boundedness principle one might reasonably expect.

The construction of the family $\cC$ is, of course, not unique. For example, one could replace \cref{eq:ourcopula} with
\begin{equation} \label{eq:generalization}
c_{f, \bm{i}}(\bu) := f\left(\sum_{j=1}^d (-1)^{i_j}u_j \Mod{1} \right)
\end{equation}
for any $\bm{i} := (i_1,\ldots,i_d) \in \{0,1\}^d$, the existing model corresponding to $\bm{i} = \bzero$. The resulting copula may no longer be exchangeable and the points at which $c_{f, \bm{i}}$ behaves erratically become harder to describe; however, most other smoothness properties (\cref{prop:pd} in particular) would still be expected to hold. Moreover, more realistic dependence models might be achieved in the general form of \cref{eq:generalization}. For instance, in the bivariate case with $i_1 \neq i_2$, some of the induced copulas will have a Spearman's $\rho$ arbitrarily close to 1, in contrast to the findings of \cref{prop:spearman}. One can show that in the bivariate case, the class $\cC$ in its current form is not closed under the $*$-product operation. We conjecture, however, that it would be closed were it extended to all copulas of the form of \cref{eq:generalization}.

We have not addressed Kendall's $\tau$ in the bivariate case. \cref{prop:spearman} provides upper and lower bounds for $\rho_{C_f}$ which are asymptotically sharp as $f$ approaches certain point masses. The lower bound of $\rho_{C} \geq -1$ is attained if and only if $\tau_{C} = -1$, because in this case $C$ is the countermonotonicity copula \citep{embrechts2001modelling}. However, from a well known inequality of \cite{daniels1950rank}, the upper bound $\rho_{C_f} \leq 1/2$ implies only that $\tau_{C_f} \leq 2/3$, and it would be interesting to know if this bound is sharp (i.e., if there exists a sequence of univariate densities $\{f_n\}$ such that $\kappa_{C_{f_n}} \to 2/3$). More generally, \cref{prop:spearman} shows that $\rho_{C_f}$ is a linear function of the second factorial moment of $X \sim f$, and a similar representation of $\tau_{C_f}$ in terms of $\EE\left[g(X)\right]$ for some function $g$ would yield interesting insights.


\begin{thebibliography}{}

\bibitem[\protect\citeauthoryear{Acz{\'e}l}{Acz{\'e}l}{1966}]{aczel1966lectures}
Acz{\'e}l, J. (1966).
\newblock {\em Lectures on functional equations and their applications}.
\newblock Academic press.

\bibitem[\protect\citeauthoryear{Berghaus, B{\"u}cher, and Volgushev}{Berghaus
  et~al.}{2017}]{berghaus2017weak}
Berghaus, B., A.~B{\"u}cher, and S.~Volgushev (2017).
\newblock Weak convergence of the empirical copula process with respect to
  weighted metrics.
\newblock {\em Bernoulli\/}~{\em 23\/}(1), 743--772.

\bibitem[\protect\citeauthoryear{Bouezmarni, Ghouch, and Taamouti}{Bouezmarni
  et~al.}{2013}]{bouezmarni2013bernstein}
Bouezmarni, T., E.~Ghouch, and A.~Taamouti (2013).
\newblock Bernstein estimator for unbounded copula densities.
\newblock {\em Statistics \& Risk Modeling\/}~{\em 30\/}(4), 343--360.

\bibitem[\protect\citeauthoryear{Daniels}{Daniels}{1950}]{daniels1950rank}
Daniels, H. (1950).
\newblock Rank correlation and population models.
\newblock {\em Journal of the Royal Statistical Society. Series B
  (Methodological)\/}~{\em 12\/}(2), 171--191.

\bibitem[\protect\citeauthoryear{Durante and Sempi}{Durante and
  Sempi}{2016}]{durante2016principles}
Durante, F. and C.~Sempi (2016).
\newblock {\em Principles of copula theory}, Volume 474.
\newblock CRC press Boca Raton.

\bibitem[\protect\citeauthoryear{Embrechts, Lindskog, and McNeil}{Embrechts
  et~al.}{2001}]{embrechts2001modelling}
Embrechts, P., F.~Lindskog, and A.~McNeil (2001).
\newblock Modelling dependence with copulas.
\newblock {\em Rapport technique, D{\'e}partement de math{\'e}matiques,
  Institut F{\'e}d{\'e}ral de Technologie de Zurich, Zurich\/}~{\em 14}, 1--50.

\bibitem[\protect\citeauthoryear{Hofert, M{\"a}chler, and McNeil}{Hofert
  et~al.}{2012}]{hofert2012likelihood}
Hofert, M., M.~M{\"a}chler, and A.~J. McNeil (2012).
\newblock Likelihood inference for archimedean copulas in high dimensions under
  known margins.
\newblock {\em Journal of Multivariate Analysis\/}~{\em 110}, 133--150.

\bibitem[\protect\citeauthoryear{Joe}{Joe}{1997}]{joe1997multivariate}
Joe, H. (1997).
\newblock {\em Multivariate models and multivariate dependence concepts}.
\newblock CRC press.

\bibitem[\protect\citeauthoryear{Levy}{Levy}{2012}]{levy2012basic}
Levy, A. (2012).
\newblock {\em Basic set theory}.
\newblock Courier Corporation.

\bibitem[\protect\citeauthoryear{McNeil and Ne{\v{s}}lehov{\'a}}{McNeil and
  Ne{\v{s}}lehov{\'a}}{2009}]{mcneil2009multivariate}
McNeil, A.~J. and J.~Ne{\v{s}}lehov{\'a} (2009).
\newblock Multivariate archimedean copulas, d-monotone functions and
  $\ell_1$-norm symmetric distributions.
\newblock {\em The Annals of Statistics\/}~{\em 37\/}(5B), 3059--3097.

\bibitem[\protect\citeauthoryear{Nelsen}{Nelsen}{2007}]{nelsen2007introduction}
Nelsen, R.~B. (2007).
\newblock {\em An introduction to copulas}.
\newblock Springer Science \& Business Media.

\bibitem[\protect\citeauthoryear{Radulovi{\'c}, Wegkamp, Zhao,
  et~al.}{Radulovi{\'c} et~al.}{2017}]{RWZ17}
Radulovi{\'c}, D., M.~Wegkamp, Y.~Zhao, et~al. (2017).
\newblock Weak convergence of empirical copula processes indexed by functions.
\newblock {\em Bernoulli\/}~{\em 23\/}(4B), 3346--3384.

\bibitem[\protect\citeauthoryear{Schweizer}{Schweizer}{1991}]{schweizer1991thirty}
Schweizer, B. (1991).
\newblock Thirty years of copulas.
\newblock In {\em Advances in probability distributions with given marginals},
  pp.\  13--50. Springer.

\bibitem[\protect\citeauthoryear{Segers}{Segers}{2012}]{segers2012asymptotics}
Segers, J. (2012).
\newblock Asymptotics of empirical copula processes under non-restrictive
  smoothness assumptions.
\newblock {\em Bernoulli\/}~{\em 18\/}(3), 764--782.

\bibitem[\protect\citeauthoryear{Segers, Sibuya, and Tsukahara}{Segers
  et~al.}{2017}]{segers2017empirical}
Segers, J., M.~Sibuya, and H.~Tsukahara (2017).
\newblock The empirical beta copula.
\newblock {\em Journal of Multivariate Analysis\/}~{\em 155}, 35--51.

\bibitem[\protect\citeauthoryear{Sklar}{Sklar}{1959}]{sklar1959fonctions}
Sklar, M. (1959).
\newblock Fonctions de r{\'e}partition {\`a} {$n$} dimensions et leurs marges.
\newblock {\em Publ. inst. statist. univ. Paris\/}~{\em 8}, 229--231.

\end{thebibliography}

\end{document}